\documentclass[11pt]{amsart}

\usepackage{amsfonts, amsmath, amscd}
\usepackage[psamsfonts]{amssymb}

\usepackage{amssymb}

\usepackage{pb-diagram}

\usepackage[all,cmtip]{xy}

\usepackage[usenames]{color}

\newtheorem{theorem}{Theorem}[section]
\newtheorem{lemma}[theorem]{Lemma}

\newtheorem{question}[theorem]{Question}
\newtheorem{remark}[theorem]{Remark}

\newtheorem{fact}[theorem]{Fact}

\numberwithin{equation}{section}

\newcommand{\NN}{\mathbb{N}}
\newcommand{\ZZ}{\mathbb{Z}}

\newcommand{\TTT}{\mathcal{T}}

\newcommand{\IR}{\mathbb{R}}

\newcommand{\Ss}{\mathbb{S}}

\newcommand{\Ff}{\mathfrak{F}}

\newcommand{\e}{\varepsilon}
\newcommand{\cl}{\mathrm{cl}}

\renewcommand{\phi}{\varphi}

\newcommand{\U}{\mathcal U}

\newcommand{\supp}{\mathrm{supp}}

\input xy
\xyoption{all}

\title[On the Mackey problem for free abelian topological groups]{On the Mackey problem for \\ free abelian topological groups}

\author[S.~Gabriyelyan]{Saak Gabriyelyan}
\address{Department of Mathematics, Ben-Gurion University of the
Negev, Beer-Sheva, P.O. 653, Israel}
\email{saak@math.bgu.ac.il}



\begin{document}

\begin{abstract}
Recently  Au\ss enhofer  and the author independently have shown that  the free abelian topological group $A(\mathbf{s})$ over a convergent sequence $\mathbf{s}$  does not admit the strongest compatible locally quasi-convex group topology that gives the first example of a locally quasi-convex abelian group without a Mackey group topology. In this note we considerably extend this example by showing that the free abelian topological group $A(X)$ over a non-discrete zero-dimensional metrizable space $X$ does not have a Mackey group topology. In particular, for every countable non-discrete metrizable space $X$, the group $A(X)$ does not have a Mackey group topology.
\end{abstract}

\maketitle


\section{Introduction}


Let $(E,\tau)$ be a locally convex space. A locally convex vector topology $\nu$ on $E$ is called {\em compatible with $\tau$} if the spaces $(E,\tau)$ and $(E,\nu)$ have the same topological dual space. The classical  Mackey--Arens theorem states that for every  locally convex space $(E,\tau)$ there exists a finest locally convex vector space topology $\mu$ on $E$ compatible with $\tau$. The topology $\mu$ is called the {\em Mackey topology} on $E$ associated with $\tau$, and if $\mu=\tau$, the space $E$ is called a {\em Mackey space}.

An analogous notion in the class of locally quasi-convex (lqc for short) abelian groups was introduced in \cite{CMPT}. For an abelian topological group $(G,\tau)$ we denote by $\widehat{G}$ the group of all continuous characters of $(G,\tau)$ (for all relevant definitions see the next section). Two group topologies  $\mu$ and $\nu$ on an abelian group $G$  are said to be {\em compatible } if $\widehat{(G,\mu)}=\widehat{(G,\nu)}$. Following \cite{CMPT},  an lqc  abelian group $(G,\mu)$ is called a {\em Mackey group} if for every lqc group topology $\nu$ on $G$ compatible with $\tau$  it follows that $\nu\leq\mu$. In this case the topology $\mu$ is called a {\em Mackey group topology} on $G$. An lqc abelian group $(G,\tau)$ is called a  {\em pre-Mackey group} 
if there is a Mackey group topology $\mu$ on $G$ compatible with $\tau$. Note that not  every Mackey locally convex space is a Mackey group. Indeed, answering a question posed in \cite{DMPT}, we proved in \cite{Gab-Mackey} that there are even metrizable locally convex spaces which are not Mackey groups.   We show in \cite{Gab-Cp} that the space $C_p(X)$, which is a Mackey space for every Tychonoff space $X$, is a Mackey group if and only every functionally bounded subset of $X$ is finite.  For historical remarks, references and open questions we referee the reader to \cite{Gab-Mackey,MarPei-Tar}.

A weaker notion than the notion of a Mackey group was introduced in \cite{Gab-Mackey}. Let $(G,\tau)$ be an lqc abelian group. A locally quasi-convex group topology $\mu$ on $G$ is called {\em quasi-Mackey} if $\mu$ is compatible with $\tau$ and there is no locally quasi-convex group topology  $\nu$ on $G$ compatible with $\tau$ such that $\mu<\nu$. The group $(G,\tau)$  is {\em quasi-Mackey} if $\tau$ is a quasi-Mackey group topology. Proposition 2.6 of \cite{Gab-Mackey} states that every locally quasi-convex abelian group has quasi-Mackey group topologies.

One of the most important classes of lqc groups is the class of free abelian topological groups. Following \cite{Mar}, an abelian topological group $A(X)$ is called {\em  the free abelian topological  group} over a Tychonoff space $X$ if $A(X)$ satisfies the following conditions:
\begin{enumerate}
\item[{\rm ($\alpha$)}] $X$ is  a subspace of $A(X)$;
\item[{\rm ($\beta$)}] any continuous map $f$ from $X$ into any abelian topological group $H$ extends uniquely to a continuous homomorphism ${\bar f}: A(X) \to H$.
\end{enumerate}
For every Tychonoff space $X$, the free abelian topological  group $A(X)$ exists, is unique up to isomorphism of abelian topological groups, see \cite{Mar}. Further, $A(X)$ is algebraically the free abelian group $\ZZ(X)$ on $X$.

The Mackey--Arens theorem suggests the following general  question posed in \cite{CMPT}: {\em Is every locally quasi-convex  abelian group a pre-Mackey  group}? Being motivated by this question we ask in Questions 4.2 and 4.4 of \cite{Gab-Mackey}:
\begin{question}
For which Tychonoff spaces $X$ the free abelian group $A(X)$ is Mackey? Is it true that $A(\mathbf{s})$, where $\mathbf{s}$ is a convergent sequence, is a Mackey group?
\end{question}

Only very recently,  Au\ss enhofer \cite{Aus3} and the author \cite{Gabr-A(s)-Mackey} independently have proved the following theorem.
\begin{theorem}[\cite{Aus3,Gabr-A(s)-Mackey}] \label{t:A(s)-Mackey}
The group $A(\mathbf{s})$ is not a pre-Mackey group.
\end{theorem}

This somewhat surprising example of an lqc group without a Mackey group topology motivates us to consider some natural {\em classes} of spaces $X$ for which  the group $A(X)$  may not have a Mackey group topology. Noting that $\mathbf{s}$ is a zero-dimensional  metrizable space one can ask whether $A(X)$ is not a pre-Mackey group for every non-discrete zero-dimensional  metrizable  space $X$. Essentially using ideas from \cite{Gabr-A(s)-Mackey,Gabr-L(X)-Mackey},  we answer this question in the affirmative that significantly generalizes Theorem \ref{t:A(s)-Mackey}.
\begin{theorem} \label{t:Mackey-space-A(X)}
If $X$ is a non-discrete zero-dimensional  metrizable  space, then $A(X)$ is neither a pre-Mackey group nor a quasi-Mackey group.
\end{theorem}
In particular, for every non-discrete countable  metrizable  space $X$, the group $A(X)$ does not have a Mackey group topology.


\section{Proof of Theorem \ref{t:Mackey-space-A(X)}}


Set $\NN:=\{ 1,2,\dots\}$. For a subset $A$ of an abelian group $G$, set $(1)A:=A$ and $(n+1)A:= (n)A+A$ for $n\in\NN$. If $g$ is an element of $G$, we denote by $\langle g\rangle$ the subgroup of $G$ generated by $g$.
Denote by $\mathbb{S}$ the unit circle multiplicative group  and set $\Ss_+ :=\{z\in  \Ss:\ {\rm Re}(z)\geq 0\}$. For every $n\in\NN$, set $\Ss_n:=\big\{ e^{2\pi i \phi}\in \Ss: \phi\in (-\frac{1}{4n}, \frac{1}{4n})\big\}$.

Recall that a topological space $X$ is called {\em zero-dimensional} if $X$ is a nonempty $T_1$-space and has a base consisting of open-and-closed sets. Every zero-dimensional space is Tychonoff, and every countable Tychonoff space is zero-dimensional by Corollary 6.2.8 of \cite{Eng}.

Let $G$ be an abelian topological group.  A character $\chi\in \widehat{G}$  is a continuous homomorphism from $G$ into $\mathbb{S}$.
A subset $A$ of $G$ is called {\em quasi-convex} if for every $g\in G\setminus A$ there exists   $\chi\in \widehat{G}$ such that $\chi(x)\notin \Ss_+$ and $\chi(A)\subseteq \Ss_+$.
An abelian topological group $G$ is called {\em locally quasi-convex} if it admits a neighborhood base at the neutral element $0$ consisting of quasi-convex sets. It is well known that the class of locally quasi-convex abelian groups is closed under taking products and subgroups.

For every Tychonoff space $X$, the group $A(X)$ is a subgroup of the free locally convex space $L(X)$ over $X$ by \cite{Tkac,Usp2}, and hence $A(X)$ is a locally quasi-convex group. If $g(x)$ is a continuous function from $X$ to $\IR$ and $\chi=a_1 x_1 +\cdots + a_n x_n \in A(X)$, set
\[
\chi(g) := a_1 g(x_1) +\cdots + a_n g(x_n ).
\]

Let $X$ be a set and let $G$ be an abelian topological group. Then the sets of the form $V^X$, where $V$ is an open neighborhood of $0\in G$, form a base at zero of the group topology $\mathfrak{u}_X$ on the direct  product $G^X$. The topology $\mathfrak{u}_X$ is called the {\em uniform topology} on $G^X$. Set $\Ff^X(G):= \big( G^X, \mathfrak{u}_X\big)$. It is easy to see that if $G$ is locally quasi-convex and $V$ is a quasi-convex neighborhood of $0\in G$, then $V^X$ is a quasi-convex neighborhood of zero in $\Ff^X(G)$ and hence $\Ff^X(G)$ is also a locally quasi-convex group.

Let $X$ be a Tychonoff space. Denote by $C(X,\Ss)$ the space of all continuous functions from $X$ to $\Ss$. Since $X$ is a closed subspace of $A(X)$ (see \cite{Mar}) and algebraically generates $A(X)$, the following well-known fact is an immediate corollary of the definition of $A(X)$.
\begin{fact} \label{p:dual-A(X)}
The restriction map $R: \widehat{A(X)} \to C(X,\Ss), F\mapsto F|_X$, is an algebraic isomorphism.
\end{fact}

We need the following lemma.
\begin{lemma}[\cite{Gabr-A(s)-Mackey}] \label{l:lemma-F0-Mackey}
Let $z,w\in\Ss$ and let $z$ have infinite order. Let $V$ be a neighborhood of $1$ in $\Ss$. If $w^l=1$ for every $l\in\NN$ such that $z^l\in V$, then $w=1$.
\end{lemma}

In the proof of Theorem \ref{t:Mackey-space-A(X)-1} below we use the following result, see Proposition 3.11 of \cite{CMPT} or Theorem 2.7 of  \cite{Gab-Mackey}.
\begin{theorem}[\cite{CMPT,Gab-Mackey}] \label{t:Char-Mackey}
For a locally quasi-convex abelian group $(G,\tau)$ the following assertions are equivalent:
\begin{enumerate}
\item[{\rm (i)}] the group  $(G,\tau)$ is  pre-Mackey;
\item[{\rm (ii)}] $\tau_1\vee\tau_2$ is compatible with $\tau$ for every locally quasi-convex group topologies $\tau_1$ and $\tau_2$ on $G$ compatible with $\tau$.
\end{enumerate}
\end{theorem}

The following theorem plays a crucial role in the proof Theorem \ref{t:Mackey-space-A(X)}.

\begin{theorem} \label{t:Mackey-space-A(X)-1}
Let $X$ be a zero-dimensional space. Assume that there exists a family $\U=\{ U_i: i\in\kappa\}$ of pairwise disjoint closed-and-open subsets of $X$ and a non-isolated point $x_\infty\in X$ such that $\bigcup\U = X\setminus\{ x_\infty\}$. Then $A(X)$ is neither a pre-Mackey group nor a quasi-Mackey group.
\end{theorem}

\begin{proof}
For every $i\in\kappa$, let $g_i$ be the characteristic function of $U_i$. Clearly, the families $\U$ and $\mathcal{G}=\{ g_i: i\in\kappa\}$ and the point $x_\infty$ satisfy the following conditions
\begin{enumerate}
\item[{\rm (i)}] $\supp(g_i) \subseteq U_i $ for every $i\in\kappa$;
\item[{\rm (ii)}] $U_i\cap U_j=\emptyset $ for all distinct $i,j\in\kappa$;
\item[{\rm (iii)}] $x_\infty\not\in U_i$ for every $i\in\kappa$ and $x_\infty\in \cl\big(\bigcup_{i\in\kappa} \{ x\in X: g_i(x)= 1\}\big)$;
\item[{\rm (iv)}]  $\mathcal{G}$ is equicontinuous at each point $y\in X\setminus\{ x_\infty\}$, i.e., for every $\e>0$ there is a neighborhood $\mathcal{O}$ of $y$ (if $y\in U_i$, take $\mathcal{O}=U_i$)  such that $|g_i(x)-g_i(y)|<\e$ for every $x\in \mathcal{O}$ and all $i\in\kappa$;
\item[{\rm (v)}] $\bigcup\U = X\setminus\{ x_\infty\}$.
\end{enumerate}

Now we construct a family
\[
\{ \TTT_z : z\in\Ss \mbox{ has infinite order}\}
\]
of  locally quasi-convex group topologies on $\ZZ(X)$ compatible with the topology $\tau$ of $A(X)$ (recall that $\ZZ(X)$ is the free abelian group generated by $X$, so $\ZZ(X)$ is the underlying abelian group of $A(X)$). To this end, we use an idea similar to the idea described in the proof of Theorem 1.3 of \cite{Gabr-A(s)-Mackey} by modulo Proposition 2.4 of \cite{Gabr-L(X)-Mackey}.

Let $z\in\Ss$ be of infinite order. Define the following algebraic monomorphism $T_z:\ZZ(X)\to A(X)\times \Ff^X(\Ss)$ by
\begin{equation} \label{equ:Mackey-Free-A(X)-1}
T_z\big( \chi \big):= \bigg( \chi, \Big( z^{\chi(g_i)}\Big)\bigg), \quad \forall \; \chi\in \ZZ(X).
\end{equation}
Denote by $\TTT_z$ the topology on $\ZZ(X)$ which is the inverse image under the mapping $T_z$ of the product topology of $A(X)\times \Ff^X(\Ss)$. It is a locally quasi-convex group topology.

\medskip
{\em Claim 1. The topology $\TTT_z$ is compatible with $\tau$.}

\smallskip
Indeed, set $G:= \big(\ZZ(X), \TTT_z\big)$. Since $\ZZ(X)$ is algebraically generated by $X$, each $F\in \widehat{G}$ is uniquely defined by its values on $X$. Therefore, by Fact \ref{p:dual-A(X)}, to show that  $\TTT_z$ is compatible with $\tau$ it is sufficient to prove that, for every $F\in \widehat{G}$, the restriction $F|_X$ of $F$ onto $X$ belongs to $C(X,\Ss)$. Fix arbitrarily $F\in \widehat{G}$.

\smallskip
{\em Claim 1.1. $F|_X$ is continuous at every point $y\in X\setminus \{ x_\infty\}$.}

\smallskip

Fix arbitrarily $n\in\NN$. Since $F$ is $\TTT_z$-continuous,  there exists a standard neighborhood $W=T_z^{-1}\big(U\times V^X\big)$ of zero in $G$, where $U$ is a neighborhood of zero in $A(X)$ and $V$ is a neighborhood of $1$ in $\Ss$, such that $F(W)\subseteq \Ss_n$. By (iv), choose a neighborhood $\mathcal{O}_n$ of $y$ in $X$ such that $x-y \in U$ and $z^{g_i(x)-g_i(y)}\in V$ for every $x\in \mathcal{O}_n$. Then, for every $x\in \mathcal{O}_n$, we obtain
\[
T_z (x-y)=\bigg( x-y, \Big( z^{(x-y)(g_i)}\Big)\bigg)= \bigg( x-y, \Big( z^{g_i(x)-g_i(y)}\Big)\bigg)\in U\times V^X.
\]
Therefore $x-y\in W$ and hence $F(x)-F(y) \in \Ss_n$ for every $x\in \mathcal{O}_n$. Thus $F|_X$ is continuous at $y$.

\smallskip
{\em Claim 1.2. $F|_X$ is continuous at  $x_\infty$.}

\smallskip
Suppose for a contradiction that $F|_X$ is not continuous at  $x_\infty$. Then there is a $p\in\NN$ such that for every neighborhood $\mathcal{O}$ of $x_\infty$ there exists a point $x_\mathcal{O} \in \mathcal{O}$ such that
\begin{equation} \label{equ:Mackey-Free-A(X)-2}
F(x_\mathcal{O})\not\in F(x_\infty)\cdot \Ss_p.
\end{equation}

{\em Subclaim 1.2. There is a point $w \in K:=\Ss\setminus ( F(x_\infty)\cdot \Ss_p )$ such that for every neighborhood $\mathcal{O}$ of $x_\infty$ and every $n\in\NN$ there is a point $x_{\mathcal{O},n}\in \mathcal{O}$ such that $F(x_{\mathcal{O},n})\in w\cdot\Ss_n$.}  Indeed, suppose for a contradiction that for every point $w \in K$ there is a neighborhood $\mathcal{O}_w$ of $x_\infty$ and an $n_w \in\NN$ such that $F(x)\not\in w\cdot\Ss_{n_w}$ for every $x\in \mathcal{O}_w$. Since $K$ is compact, there are $w_1,\dots, w_b\in K$ and corresponding neighborhoods $\mathcal{O}_{w_1},\dots,\mathcal{O}_{w_b}$ of $x_\infty$ and natural numbers $n_{w_1}, \dots, n_{w_b}$ such that
\begin{enumerate}
\item[($\dagger$)] the sets $w_i \cdot \Ss_{n_{w_i}}$ cover $K$, and
\item[($\ddagger$)] $F(x) \not\in w_i \cdot \Ss_{n_{w_i}}$ for each $x\in \mathcal{O}_{w_i}$ and  every $i=1,\dots,b$.
\end{enumerate}
Set $\mathcal{O}:=\bigcap_{i=1}^b \mathcal{O}_{w_i}$. Then, for every $x\in \mathcal{O}$, ($\dagger$) and ($\ddagger$) imply that $F(x) \in \Ss\setminus K=F(x_\infty)\cdot \Ss_p$. But this contradicts (\ref{equ:Mackey-Free-A(X)-2}). The subclaim is proved.

\smallskip
Now fix arbitrarily a point $w\in K$ satisfying the condition of Subclaim 1.2. Since $F$ is $\TTT_z$-continuous,  there exists a standard neighborhood $W=T_z^{-1}\big(U\times V^X\big)$ of zero in $G$, where $U$ is a neighborhood of zero in $A(X)$ and $V$ is an open neighborhood of $1$ in $\Ss$, such that $F(W)\subseteq \Ss_+$.  Observe that, by (\ref{equ:Mackey-Free-A(X)-1}),  $\chi\in W$ if and only if
\begin{equation} \label{equ:Mackey-Free-A(X)-3}
\chi\in U \; \mbox{ and } \; z^{\chi (g_i)} \in V \mbox{ for every } i\in \kappa.
\end{equation}
Set $w_0:=w\cdot F(x_\infty)^{-1}$, so $w_0\not\in \Ss_p$ and hence $w_0\not= 1$. We assume additionally that $w_0\not\in V$. Set $L:=\{ l\in\NN: z^l\in V\}$.  Since $\langle z\rangle$ is dense in $\Ss$, the set $L$ is not empty. We distinguish between two cases.

\medskip
{\em Subcase 1.2.A. Assume that  $w_0^l =1$ for every $l\in L$.} Then Lemma \ref{l:lemma-F0-Mackey} implies $w_0=1$. Since $w_0\not=1 $ we obtain that this case is impossible.

\medskip
{\em Subcase 1.2.B. There is an $l_0\in L$ such that $w_0^{l_0}\not= 1$.} Then there exists a $t\in\NN$ such that $w_0^{l_0 t}\not\in \Ss_+$.
Choose a closed-and-open neighborhood $\mathcal{O}_0$ of $x_\infty$ in $X$ such that
\begin{equation} \label{equ:Mackey-Free-A(X)-4}
(l_0 t)(\mathcal{O}_0 - \mathcal{O}_0) \subseteq U.
\end{equation}
Take $M\in\NN$ such that
\begin{equation} \label{equ:Mackey-Free-A(X)-31}
(w_0\cdot \Ss_{M})^{l_0 t} \cap \Ss_+ =\emptyset.
\end{equation}

Now Subclaim 1.2 implies that there is a point $x_1\in\mathcal{O}_0$ such that $F(x_1)\in w \cdot \Ss_M$. By  (ii) and (v), choose a unique index $i_1\in\kappa$ such that $x_1 \in U_{i_1} $. For a closed-and-open neighborhood $\mathcal{O}_1 := \mathcal{O}_0 \setminus U_{i_1}$ of $x_\infty$ (see (iii)), by Subclaim 1.2, choose a point $x_2\in\mathcal{O}_1$ such that $F(x_2)\in w \cdot \Ss_M$. By  (ii) and (v), choose a unique index $i_2\in\kappa$ such that $x_2 \in U_{i_2} $. Clearly, $i_2 \not= i_1$. Continuing this process we can find points $x_1,\dots,x_t\in \mathcal{O}_0$ such that $x_\alpha\in U_{i_\alpha}=\{ x\in X: g_{i_\alpha}(x)=1\}$ for some pairwise distinct indices $i_1,\dots,i_t \in\kappa$ and
\begin{equation} \label{equ:Mackey-Free-A(X)-32}
F(x_{\alpha})\in w\cdot\Ss_M, \quad \mbox{ for every } \; \alpha=1,\dots,t.
\end{equation}

Set
\[
\eta:=x_\infty + l_0 (x_1 -x_\infty) +\cdots + l_0 (x_t -x_\infty) \in A(X).
\]
Then, by (\ref{equ:Mackey-Free-A(X)-4}), $\eta -x_\infty \in U$. By the choice of $x_1,\dots,x_t$ and (i) and (ii),  for every $j\in \kappa$, we have
\begin{equation} \label{equ:Mackey-Free-A(X)-5}
 g_j(x_\alpha) = 1 \mbox{ if } i_\alpha =j, \mbox{ and } g_j(x_\alpha) = 0 \mbox{ if } i_\alpha \not=j, \quad (\alpha=1,\dots,t)
\end{equation}
and therefore we obtain
\[
\begin{aligned}
z^{(\eta -x_\infty)(g_j)} & =z^{\big(l_0 (x_1 -x_\infty) +\cdots + l_0 (x_t -x_\infty)\big) (g_j)} \stackrel{(iii)}{=} z^{l_0\cdot g_j(x_1)  +\cdots + l_0\cdot g_j(x_t) } \\
& \stackrel{(\ref{equ:Mackey-Free-A(X)-5})}{=} \left\{ \begin{aligned}
            1, & \mbox{ if } j\not\in \{ i_1,\dots,i_t\} \\
            z^{l_0 \cdot g_j(x_\alpha) }=z^{l_0},  & \mbox{ if } j =i_\alpha \mbox{ for some } 1\leq\alpha\leq t.
            \end{aligned} \right.
\end{aligned}
\]
Since $z^{l_0} \in V$ we obtain that $z^{(\eta -x_\infty)(g_j)}$ belongs to $V$ for every index $j\in \kappa$.  Thus $\eta -x_\infty \in W$ and hence $F(\eta -x_\infty) \in \Ss_+$. On the other hand,
\[
\begin{aligned}
F(\eta -x_\infty) & =\prod_{\alpha=1}^t F(x_\alpha -x_\infty)^{l_0} =\prod_{\alpha=1}^t \left( F(x_\alpha)\cdot F(x_\infty)^{-1}\right)^{l_0} \quad \big( \mbox{by } (\ref{equ:Mackey-Free-A(X)-32})\big)\\
& \in \prod_{\alpha=1}^t \left(  w \cdot \Ss_M \cdot F(x_\infty)^{-1}\right)^{l_0} =(w_0\cdot \Ss_{M})^{l_0 t}  .
\end{aligned}
\]
Therefore, by (\ref{equ:Mackey-Free-A(X)-31}), $F(\eta -x_\infty) \not\in \Ss_+$. This contradiction shows that $F$ must be continuous also at $x_\infty$.

Now Claims 1.1 and 1.2 imply that $F\in \widehat{A(X)}$. Thus the topology $\TTT_z$ is compatible with the topology $\tau$ of $A(X)$.

\medskip
{\em Claim 2. For every element $a\in\Ss\setminus\{ 1\}$ of finite order, the topology $\TTT_z \vee \TTT_{az}$ is not compatible with $\tau$.} Indeed, let $r$ be the order of $a$ and set
\[
D_r := \left\{ \chi \in \ZZ(X) : \chi(g_i) \in r\ZZ \quad \forall i\in\kappa \right\}.
\]
Consider standard neighborhoods of zero
\[
W_z=T_z^{-1}\big(U\times V^X\big)  \;\; \mbox{ and } \;\; W_{az}=T_{az}^{-1}\big(U\times V^X\big)
\]
in $\TTT_z$ and $\TTT_{az}$, respectively, where $U\in\tau$ and $V$ is a symmetric open neighborhood  of $1$ in $\Ss$ such that $V\cdot V\cap \langle a\rangle =\{ 1\}$. Then, by (\ref{equ:Mackey-Free-A(X)-3}), we have
\[
W_z \cap W_{az} = \left\{ \chi\in \ZZ(X) : \chi\in U \mbox{ and } z^{\chi(g_i)}, (az)^{\chi(g_i)} \in V \mbox{ for every } i\in\kappa\right\}.
\]
We show that $W_z \cap W_{az} \subseteq D_r$. Indeed, if $\chi\in W_z \cap W_{az}$, then $a^{\chi(g_i)}\in V\cdot V$, and hence $a^{\chi(g_i)}=1$ for every $i\in\kappa$. Therefore, for every $i\in\kappa$, there is an $s_i\in\ZZ$ such that $\chi(g_i) = s_i \cdot r$. Thus $W_z \cap W_{az} \subseteq D_r$.

By (ii), for every $\chi\in \ZZ(X)$, the sum $\sum_i \chi(g_i)$ is finite. Therefore the map
\[
F_0(\chi):= a^{\sum_i \chi(g_i)},\quad \chi\in \ZZ(X),
\]
is a well-defined algebraic homomorphism from $\ZZ(X)$ into $\Ss$. As $F_0(x)= a^{\sum_i g_i(x)}$ for every $x\in X$, (i)-(iii) imply that for every neighborhood $\mathcal{O}$ of $x_\infty$ in $X$ there is a point $x_\mathcal{O}\in \mathcal{O}$ such that  $\sum_i g_i(x_\mathcal{O})=1$. Therefore $F_0(x_\mathcal{O})=a \not= 1=F_0(x_\infty)$, and hence $F_0|_X$ is not continuous at $x_\infty$. Hence $F_0\not\in \widehat{A(X)}$ by Fact \ref{p:dual-A(X)}. On the other hand,
\[
F_0(W_z \cap W_{az} )\subseteq F_0(D_r) =\{ 1\}.
\]
As $W_z \cap W_{az}\in \TTT_z \vee \TTT_{az}$ it follows that $F_0$ is $\TTT_z \vee \TTT_{az}$-continuous. Thus $\TTT_z \vee \TTT_{az}$ is not compatible with $\tau$.

\medskip
{\em Claim 3. $\tau <\TTT_z$, so $\tau$ is not quasi-Mackey.} By (\ref{equ:Mackey-Free-A(X)-1}), it is clear that $\tau \leq\TTT_z$. To show that $\tau\not= \TTT_z$, suppose for a contradiction that $\TTT_z =\tau$. Then, by Claim 1, $\TTT_z \vee \TTT_{az}=\tau \vee \TTT_{az}= \TTT_{az}$ is compatible with $\tau$. But this contradicts Claim 2.

\medskip
{\em Claim 4. The group $A(X)$ is not pre-Mackey}. This immediately follows from Claim 2  and Theorem \ref{t:Char-Mackey}. The theorem is proved.
\end{proof}

Now we are ready to prove Theorem \ref{t:Mackey-space-A(X)}.

\begin{proof}[Proof of Theorem \ref{t:Mackey-space-A(X)}]
Take arbitrarily a non-isolated point $x_\infty\in X$ and let $\{ V_n: n\in\NN\}$ be a strictly decreasing base at $x_\infty$ consisting of closed-and-open sets in which $V_1 =X$. It is clear that the family  $\{ V_n\setminus V_{n+1}: n\in\NN\}$ and the point $x_\infty$ satisfy conditions of Theorem  \ref{t:Mackey-space-A(X)-1}. Thus $A(X)$ is neither a pre-Mackey group nor a quasi-Mackey group.
\end{proof}

\begin{remark} {\em
Note that Theorem  \ref{t:Mackey-space-A(X)-1} can be applied also for non-metrizable spaces. Indeed, let $X$ be an arbitrary uncountable set and $x_\infty$ be a point of $X$. Consider the following topology on $X$: every $x\in X\setminus \{ x_\infty\}$ is isolated and a neighborhood base at $x_\infty$ consists of the sets of the form $(X\setminus A)\cup \{ x_\infty\}$, where $A$ is a countable subset of $X$. Then $X$ is a Lindel\"{o}f zero-dimensional non-metrizable space. It is clear that the family $\big\{ \{ x\}: x \in X\setminus \{ x_\infty\} \big\}$ and the point $x_\infty$ satisfy the conditions of Theorem  \ref{t:Mackey-space-A(X)-1}. Thus $A(X)$ does not have a Mackey group topology.}
\end{remark}

We finish with the following question.

\begin{question}
Does there exist a non-discrete (metrizable or compact) Tychonoff space $X$ for which the group $A(X)$ admits a Mackey group topology?
\end{question}


\end{document}